\documentclass[12pt,a4paper,reqno]{amsart}
\allowdisplaybreaks
\usepackage{amsmath}
\usepackage{amsfonts}
\usepackage{amssymb,amsthm,amsfonts,amsthm,latexsym,enumerate,url,cases}
\numberwithin{equation}{section}
\usepackage{mathrsfs}
\usepackage{hyperref}
\hypersetup{colorlinks=true,citecolor=blue,linkcolor=blue,urlcolor=blue}
\def\pmod #1{\ ({\rm{mod}}\ #1)}

\theoremstyle{plain}
\newtheorem{theorem}{Theorem}

\newtheorem{lemma}{Lemma}

\newtheorem{corollary}{Corollary}
\newtheorem{proposition}{Proposition}

\theoremstyle{definition}

\usepackage{etoolbox}
\makeatletter
\patchcmd{\@settitle}{\uppercasenonmath\@title}{}{}{}
\patchcmd{\@setauthors}{\MakeUppercase}{}{}{}
\patchcmd{\section}{\scshape}{}{}{}
\makeatother

\begin{document}

\title
[{On a conjecture on shifted primes with large prime factors, II}]
{On a conjecture on shifted primes with large prime factors, II}

\author
[Y. Ding] 
{Yuchen Ding}

\address{(Yuchen Ding) School of Mathematical Sciences,  Yangzhou University, Yangzhou 225002, People's Republic of China}
\email{ycding@yzu.edu.cn}


\begin{abstract}
Let $\mathcal{P}$ be the set of primes and $\pi(x)$ the number of primes not exceeding $x$. Let also $P^+(n)$ be the largest prime factor of $n$ with convention $P^+(1)=1$ and
$$
T_c(x)=\#\left\{p\le x:p\in \mathcal{P},P^+(p-1)\ge p^c\right\}.
$$
Motivated by a 2017 conjecture of Chen and Chen, we show that for any $8/9\le c<1$
$$
\limsup_{x\rightarrow\infty}T_c(x)/\pi(x)\le 8(1/c-1),
$$
which clearly means that 
$$
\limsup_{x\rightarrow\infty}T_c(x)/\pi(x)\rightarrow 0, \quad \text{as}~c\rightarrow1.
$$
\end{abstract}
\maketitle

\section{Introduction}
The investigations on shifted primes with large prime factors drew attention to people in a brilliant article of Goldfeld \cite{Gold}. Historically, this topic had aroused great concern among the community due to its unexpected connection with the first case of Fermat's last theorem, thanks to the theorems of Fouvry \cite{Fou},  Adleman and Heath-Brown \cite{AHB}.

For any positive integer $n$, let $P^+(n)$ be the largest prime factor of $n$ with convention $P^+(1)=1$. Let $\mathcal{P}$ be the set of primes and $\pi(x)$ the number of primes not exceeding $x$. For $0<c<1$ let
$
T_c(x)=\#\left\{p\le x:p\in \mathcal{P},P^+(p-1)\ge p^c\right\}.
$
As early as 1969, Goldfeld \cite{Gold} proved 
\begin{align*}
\liminf_{x\rightarrow\infty}T_{\frac{1}{2}}(x)/\pi(x)\ge 1/2.
\end{align*}
Goldfeld further remarked that his argument also leads to another fact
\begin{equation}\label{eq1-1}
\liminf_{x\rightarrow\infty}T_{c}(x)/\pi(x)>0,
\end{equation}
provided that $c<\frac{7}{12}$. It turns out that exploring large $c$ to satisfy Eq. (\ref{eq1-1}) is rather difficult and important. For improvements on the values of $c$, see 
Motohashi \cite{Moto}, Hooley \cite{Hoo1,Hoo2}, Deshouillers--Iwaniec \cite{Iwa} and Fouvry \cite{Fou}. 
Up to now, the best numerical value of $c$ satisfying Eq. (\ref{eq1-1}), with a payoff by replacing $\pi(x)$ with $\pi(x)/\log x$, is 0.677, obtained by Baker and Harman \cite{B-H}.

In a former note \cite{Ding} of this sequel, I showed that the following inequality
\begin{align}\label{eq1}
\limsup_{x\rightarrow\infty}T_{c}(x)/\pi(x)<1/2
\end{align}
holds for some absolute constant $c<1$.
As a corollary, I disproved a 2017 conjecture of Chen and Chen \cite{CC} who conjectured that 
$$
\liminf_{x\rightarrow\infty}T_{c}(x)/\pi(x)\ge 1/2
$$
for any $1/2\le c<1$.
The proof in my former note is based on the following deep result which is a corollary of the elaborate Brun--Titchmarsh inequality. 
\begin{proposition}\label{pro1}\cite[Lemma 2.2]{Wu}
There exist two functions $K_2(\theta)>K_1(\theta)>0$, defined on the interval $(0,17/32)$ such that for each fixed real $A>0$, and all sufficiently large $Q=x^\theta$, the inequalities
$$K_1(\theta)\frac{\pi(x)}{\varphi(m)}\le \pi(x;m,1)\le K_2(\theta)\frac{\pi(x)}{\varphi(m)}$$
hold for all integers $m\in(Q,2Q]$ with at most $O\left(Q(\log Q)^{-A}\right)$ exceptions, where the implied constant depends only on $A$ and $\theta$. Moreover, for any fixed $\varepsilon>0$, these functions can be chosen to satisfy the following properties:\\
$\bullet$ $K_1(\theta)$ is monotonic decreasing, and $K_2(\theta)$ is monotonic increasing.\\
$\bullet$ $K_1(1/2)=1-\varepsilon$ and $K_2(1/2)=1+\varepsilon$.
\end{proposition}

The constant $c$ in Eq. (\ref{eq1}) is not specified therein due to the indeterminate amounts $K_1(\theta)$ in Proposition \ref{pro1}. In fact, $K_1(\theta)$ (and hence $c$) can be explicitly given if one checks carefully the articles of Baker--Harman \cite{BH} for $1/2\le\theta\le 13/25$ and Mikawa \cite{Mi} for $13/25\le \theta\le 17/32$. Actually, $K_1(\theta)\ge 0.16$ for $1/2\le\theta\le 13/25$ \cite[Theorem 1]{BH} and $K_1(\theta)\ge 1/100$ for Mikawa's range \cite[Eq. (4)]{Mi}. However, it seems that the constant $c$ in Eq. (\ref{eq1}) obtained upon this way will be very close to $1$ (see the proofs in \cite{Ding}).

In \cite{Ding}, I also pointed out that Chen and Chen's conjecture is already in contradiction with the Elliott--Halberstam conjecture according to the works of Pomerance \cite{Po}, Granville \cite{Gran}, Wang \cite{Wang} and Wu \cite{Wu}. In fact, one has
\begin{align}\label{eq2}
\limsup_{x\rightarrow\infty}T_c(x)/\pi(x)=\lim_{x\rightarrow\infty}T_c(x)/\pi(x) =\left(1-\rho\left(\frac{1}{c}\right)\right)\rightarrow 0, \quad \text{as}~c\rightarrow1
\end{align}
under the assumption of the Elliott--Halberstam conjecture, where $\rho(u)$ is the Dickman function, defined as the unique continuous solution of the equation
differential--difference
\begin{align*}
\begin{cases}
\rho(u)=1, & 0\le u\le 1,\\
u\rho'(u)=-\rho(u-1), & u>1.
\end{cases}
\end{align*} 

However, things involving with the study on Chen--Chen's conjecture as well as my prior result (i.e., Eq. (\ref{eq1})) got somewhat misleading to some extent. In fact, indicated by a former result of Erd\H os \cite[Lemma 4]{erdos}, as early as 1935, people could already conclude from his proof, by combining with Lemma \ref{lem2} of Wu (see below), that Eq. (\ref{eq2}) is true in part. That is,

\begin{theorem}[Erd\H os]\footnote{Actually, the original result of Erd\H os states that
$$ 
\#\left\{p\le (\log x)^{1+\rho}:P^+(p-1)\le \log x\right\}>c_\rho\frac{(\log x)^{1+\rho}}{\log\log x}
$$
for some given constants $\rho$ and $c_\rho$. Essentially, Theorem \ref{thm0} can be deduced from his proof of this result by adding Wu's lemma.}
\footnote{Shortly after the manuscript of this note was put on the arXiv, Brad Rodgers informed the author that, joint with Bharadwaj, they independently obtained the same result with a general form in the probabilistic language. Their article was now also on the arXiv, see \cite[Theorem 8 and Corollary 9]{BB}.}\label{thm0}
Unconditionally, we have
$$
\limsup_{x\rightarrow\infty}T_c(x)/\pi(x)\rightarrow 0, \quad \text{as}~c\rightarrow1.
$$
\end{theorem}

Since Erd\H os' conclusion, by his statements therein, is not very clear in the formulations as its appearance in Theorem \ref{thm0}, I feel that it is meaningful and necessary to restate them again to form the explicit theorem above due to the recent publication history mentioned here. It would therefore, as I believe, be of interest to push the Erd\H os theorem a little further to the following quantitative form as explicit as possible. 

\begin{theorem}\label{thm1}
For $8/9\le c<1$, we have
$$
\limsup_{x\rightarrow\infty}T_c(x)/\pi(x)\le 8\left(1/c-1\right).
$$
\end{theorem} 

We note that the restriction on $c\ge 8/9$ in our theorem is natural since otherwise the upper bound would exceed $1$ which is certainly meaningless.
Theorem \ref{thm1} can also be compared with the prior results of Goldfeld \cite{Gold}, Luca {\it et al.} \cite{LMP} and Chen--Chen \cite{CC} which state that
$$
\liminf_{x\rightarrow\infty}T_c(x)/\pi(x)\ge 1-c
$$
for any $0<c\le 1/2$. Their bounds were recently improved in part by Feng--Wu \cite{FW} and Liu--Wu--Xi \cite{LWX}.
From Theorem \ref{thm1}, we clearly have two corollaries, one of which is the Erd\H os' theorem (i.e., Theorem \ref{thm0}). The other one revisits the main result Eq. (\ref{eq1}) of my former note in a quantitative form.    

\begin{corollary}\label{cor1}
For $c>16/17$, we have
$$
\limsup_{x\rightarrow\infty}T_{c}(x)/\pi(x)<1/2.
$$
\end{corollary}

\section{Proofs}
From now on, $p$ will always be a prime. The proof of Theorem \ref{thm1} is based on the following lemma deduced from the sieve method (see e.g. \cite[page 172, Theorem 5.7]{Halberstam}). 

\begin{lemma}\label{lem1}
Let $g$ be a natural number, and let $a_i,b_i~(i=1,2,\cdot\cdot\cdot,g)$ be integers satisfying 
$$
E:=\prod_{i=1}^{g}a_i\prod_{1\le r<s\le g}(a_rb_s-a_sb_r)\neq0.
$$
Let $\rho(p)$ denote the number of solutions of 
$$
\prod_{i=1}^{g}(a_in+b_i)\equiv 0\pmod{p},
$$
and suppose that
$$
\rho(p)<p \quad \text{for~all~} p.
$$
Let $y$ and $z$ be real numbers satisfying
$$
1<y\le z.
$$ 
Then
\begin{multline*}
\left|\{n:z-y<n\le z, a_in+b_i~\text{prime~for~}i=1,2,\cdot\cdot\cdot,g\}\right|\\
\le 2^gg!\prod_{p}\left(1-\frac{\rho(p)-1}{p-1}\right)\left(1-\frac{1}{p}\right)^{-g+1}\!\!\!\frac{y}{\log ^gy} \left(1+O\left(\frac{\log\log3y+\log\log 3|E|}{\log y}\right)\right),
\end{multline*}
where the constant implied by the $O$-symbol depends at most on $g$.
\end{lemma}

We also need the following important relation established by Wu \cite[Theorem 2]{Wu}.

\begin{lemma}\label{lem2}
For $0<c<1$, let $$T'_c(x)=\#\{p\le x:p\in\mathcal{P}, P^+(p-1)\ge x^c\}.$$
Then for sufficiently large $x$ we have
$$T_c(x)=T'_c(x)+O\left(\frac{x\log\log x}{(\log x)^2}\right).$$
\end{lemma}

We now turn to the proof of Theorem \ref{thm1}.
\begin{proof}[Proof of Theorem \ref{thm1}]
Let $x$ be a sufficiently large number throughout our proof. Instead of investigating $T_c(x)$, we deal with $T'_c(x)$ firstly.
For $1/2\le c<1$ it is easy to see
\begin{align}\label{new1}
T'_c(x)=
\sum_{\substack{x^c\le q<x\\ q\in \mathcal{P}}}\sum_{\substack{p\le x\\q|p-1}}1.
\end{align}
On putting $p-1=qh$ in the sum of Eq. (\ref{new1}) and then exchanging the order of the sums, we get
\begin{align}\label{eq13-1}
T'_c(x)=\sum_{\substack{x^c\le q<x \\ q\in \mathcal{P}}}\sum_{\substack{h< x/q\\qh+1\in \mathcal{P}}}1\le \sum_{\substack{h<x^{1-c}\\ 2|h}}\sum_{\substack{2<q<x/h \\ q,qh+1\in \mathcal{P}}}1.
\end{align}
For any $h$ with $2|h$ and $h<x^{1-c}$, let $\rho(p)$ denote the number of solutions of
\begin{align*}
n(hn+1)\equiv 0\pmod{p}.
\end{align*}
Then
\begin{align*} 
\rho(p)=
\begin{cases}
1, & \text{if~}p|h,\\
2, & \text{otherwise.}
\end{cases}
\end{align*}
Now, by Lemma \ref{lem1} with $g=2, a_1=1, b_1=0, a_2=h,b_2=1$ and $z=y=x/h$ we have
$$
3|E|=3h\ll x, \quad 3y=3x/h\ll x  \quad \text{and} \quad y=x/h\ge \sqrt{x}, 
$$
from which it follows that
\begin{align}\label{eq13-2}
\sum_{\substack{2<q<x/h \\ q,qh+1\in \mathcal{P}}}1\le 16\mathfrak{S}\prod_{\substack{p|h\\ p>2}}\left(1+\frac{1}{p-2}\right)
\frac{x/h}{\log^2(x/h)} \left(1+O\left(\frac{\log\log x}{\log x}\right)\right),
\end{align}
where the empty product for $\prod_{\substack{p|h\\ p>2}}$ above denotes $1$ as usual and
$$
\mathfrak{S}=\prod_{p>2}\left(1-\frac{1}{p-1}\right)\left(1-\frac{1}{p}\right)^{-1}=\prod_{p>2}\left(1-\frac{1}{(p-1)^2}\right).
$$
Inserting Eq. (\ref{eq13-2}) into Eq. (\ref{eq13-1}) we obtain
\begin{align}\label{eq13-3}
T'_c(x)&\le (1+o(1))16\mathfrak{S}\sum_{\substack{h<x^{1-c}\\ 2|h}}\prod_{\substack{p|h\\ p>2}}\left(1+\frac{1}{p-2}\right)
\frac{x/h}{\log^2(x/h)}.
\end{align}
It can be noted that 
\begin{align}\label{eq27-1}
\prod_{\substack{p|h\\ p>2}}\left(1+\frac{1}{p-2}\right)\le 2\prod_{\substack{p|h\\ p>2}}\left(1+\frac{1}{p}\right)
\end{align}
since the gaps between odd primes are at least $2$, from which we can already give a nontrivial upper bound of $T'_c(x)$ via partial summations. To make our bound more explicit than Eq. (\ref{eq27-1}), we employ a nice result of Banks and Shparlinski \cite[Lemma 2.3]{BS} (on taking $a=1$ therein) which states that for $2\le z<x^{1-c}$,
\begin{align}\label{eq27-2}
S(z):=\sum_{\substack{h<z\\ 2|h}}\frac{1}{h}\prod_{\substack{p|h\\ p>2}}\left(1+\frac{1}{p-2}\right)=\frac{1+o(1)}{2\mathfrak{S}}\log z.
\end{align}
And for $1\le z<2$, we appoint $S(z)=0$.
We get from partial summations that
\begin{align}\label{eq27-5}
\sum_{\substack{h<x^{1-c}\\ 2|h}}\prod_{\substack{p|h\\ p>2}}\left(1+\frac{1}{p}\right)\!\frac{1/h}{\log^2(x/h)}=\frac{S\left(x^{1-c}\right)}{\left(\log x^c\right)^2}-\int_{1}^{x^{1-c}}S(z)d\left(\log \frac{x}{z}\right)^{-2}.
\end{align}
Now, routine computations provide us
$$
\frac{S\left(x^{1-c}\right)}{\left(\log x^c\right)^2}=\frac{1+o(1)}{\mathfrak{S}}\frac{(1-c)}{2c^2}(\log x)^{-1}
$$
and
\begin{align*}
\int_{1}^{x^{1-c}}S(z)d\left(\log \frac{x}{z}\right)^{-2}&=\frac{1+o(1)}{\mathfrak{S}}\int_{1}^{x^{1-c}}\frac{\log z}{z}\left(\log \frac{x}{z}\right)^{-3}dz\\
&=\frac{1+o(1)}{\mathfrak{S}}\int_{x^c}^{x}\frac{\log x-\log u}{u}(\log u)^{-3}du  \quad \quad (u=x/z)\\
&=\frac{1+o(1)}{\mathfrak{S}}\left(\frac{1-c}{2c^2}+\frac{1}{2}-\frac{1}{2c}\right)(\log x)^{-1},
\end{align*}
thanks to the estimate Eq. (\ref{eq27-2}). Putting together the formulas above, one see that the right--hand side of Eq. (\ref{eq27-5}) equals
\begin{align}\label{eq27-6}
\frac{1+o(1)}{\mathfrak{S}}\left(\frac{1}{2c}-\frac{1}{2}\right)(\log x)^{-1}.
\end{align}
Taking Eq. (\ref{eq27-6}) into Eq. (\ref{eq13-3}) we immediately obtain that
\begin{align*}
T'_c(x)&\le (1+o(1))8\left(\frac{1}{c}-1\right)\frac{x}{\log x}.
\end{align*}
Therefore, by Lemma \ref{lem2} we have
$$
T_c(x)=T'_c(x)+O\left(\frac{x\log\log x}{(\log x)^2}\right)\le (1+o(1))8\left(\frac{1}{c}-1\right)\frac{x}{\log x}.
$$
Our theorem now follows from the prime number theorem. 
\end{proof}

\section{Remarks}
Under the assumption of the Elliott--Halberstam conjecture, it is reasonable to predict that the exact value of $c$ in Corollary \ref{cor1} should be $e^{-1/2}=0.60653\cdots$ from Eq. (\ref{eq2}) and the following recursion formula (see e.g. \cite[Eq. (7.6)]{MV}) on Dickman's function
$$
\rho(v)=u-\int_{u}^{v}\frac{\rho(t-1)}{t}dt \quad (1\le u\le v).
$$
It therefore seems to be of independent interest to improve, as far as possible, the numerical value of $c$ in Corollary \ref{cor1}. We leave this as a challenge to dear readers who are interested in this. 

Though we provided nontrivial upper bounds on $T_c(x)$ for $8/9\le c<1$ in Theorem \ref{thm1} for some absolute constant $c_0$, the expansion of these bounds to $1/2\le c<1$ is, however,  an unsolved problem.

\section*{Acknowledgments}
The author is really grateful to the one who pointed out that the facts stated in footnote 1, page 2. 

The author is supported by National Natural Science Foundation of China  (Grant No. 12201544), Natural Science Foundation of Jiangsu Province, China (Grant No. BK20210784), China Postdoctoral Science Foundation (Grant No. 2022M710121).

\end{document}